\documentclass[twoside]{amsart}
\usepackage{amsmath,amsthm,amssymb}

\theoremstyle{plain}
\newtheorem{theorem}{Theorem}
\newtheorem{proposition}[theorem]{Proposition}
\newtheorem{lemma}[theorem]{Lemma}

\theoremstyle{remark}
\newtheorem{remark}[theorem]{Remark}

\newcommand{\F}{{\mathbb F}}
\newcommand{\Q}{{\mathbb Q}}
\newcommand{\R}{{\mathbb R}}
\newcommand{\Z}{{\mathbb Z}}

\def\bar{\overline}

\newcommand{\cL}{ {\mathcal L} }

\newcommand{\Supp}{ {\rm Supp} }     %Generic Fredholm polygon
\newcommand{\ord}{{\rm ord}}
\newcommand{\sgn}{{\rm sgn} }
\newcommand{\NP }{ {\rm  NP} }     %Newton polygon

\begin{document}

%Topmatter
\title[Some families of supersingular Artin-Schreier curves]
{Some families of supersingular Artin-Schreier curves in
characteristic $>2$}
\author{Hui June Zhu}
\address{Department of Mathematics, SUNY at Buffalo, Buffalo, NY 14260}
\email{zhu@cal.berkeley.edu}

\date{\today}
\subjclass{11,14}

\maketitle

\section{Introduction}

A curve over finite field is supersingular if
its Jacobian is supersingular as an abelian variety.
On the one hand, supersingular
abelian varieties form the smallest (closed)
stratum in the moduli space of abelian varieties,
on the other the intersection of
Jacobian locus and the stratification of
moduli space is little known. Consequently
it is very difficult to locate a family of
supersingular curve.  See \cite{Li-Oort:98}.
In characteristic $2$ some ground-breaking progress has been make in \cite{SZ:01}\cite{SZ:02}, where families of supersingular curves are given explicitly using some new sharp slope estimation method.
However, that method does not apply easily to cases when characteristic is not
$2$. In this paper we develop a new method
to allow us to prove supersingularity of Artin-Schreier curves in characteristic $>2$. To illustrate how our method works, we use it show

\begin{theorem}
The following two families of Artin-Schreier curves are supersingular:
$$
\begin{array}{lll}
X_1:& y^7-y=x^5+cx^2, &c\in\bar\F_7\\
X_2:& y^5-y=x^7+cx,   &c\in \bar\F_5.
\end{array}
$$
\end{theorem}

\begin{remark}
We remark that using the same technique we were able to prove
the supersingularity of the following family
$y^3-y=x^7+a x^2+ b x$ over $\bar\F_3$
(apr\`es Noam Elkies \cite{Elkies},
who is able to do so using a completely different
approach).
\end{remark}

Our method is based upon the $p$-adic Dwork method (see Section 2), and it reduces the supersingularity criterion of an Artin-Schreier curve
to some strikingly simple linear algebra computation. Proof of
Theorem 1 is done in Sections 3 and 4.
This makes the method extremely promising in locating or verifying supersingularity of more families of Artin-Schreier curves in characteristic small (relative to the genus).

Finally we remark that $X_1$ and $X_2$ stood out as
supersingular suspects via extensive computer search (via consideration of their monodromy)
by joint effort of several people: Noam Elkies, Nick Katz, Eric Rain and Michael Zieve. We thank Noam Elkies and Bjorn Poonen for passing on this ``supersingular" question to us.

\section{The Dwork trace formula}

We first recall Dwork trace formula. Let $\F_q=\F_{p^a}$ for some positive integer $a$. Let $\Omega_1:=\Q_p(\zeta_p)$ and
$\Omega_a$ its unramified extension of degree $a$.
Let $\ord_p(\cdot)$ be the $p$-adic valuation and
let $\ord_q(\cdot)$ be the normalized
$p$-adic valuation so that $\ord_qq = 1$.
Let $\tau$ be the lift of Frobenius endomorphism
$a\mapsto a^p$ of $\F_q$ to $\Omega_a$ which fixes $\Omega_1$.
For any $c\in\R_{>0}$ and $b\in \R$ let
$\cL(c,b)$ be the set of power series defined by
$$\cL(c,b):=\{
\sum_{n=0}^{\infty}A_nX^n\mid A_n\in \Q_p(\zeta_p),
\ord_pA_n\geq \frac{cn}{d}+b\}.$$
Let $\cL(c):=\bigcup_{b\in\R}\cL(c,b)$.
Let $E(x)$ be the Artin-Hasse exponential function, that is,
$E(x)=\exp(\sum_{j=0}^{\infty}\frac{x^{p^j}}{p^j})$.
Let $\gamma$ be a $p$-adic root of $\log(E(x))$ with
$\ord_p\gamma=\frac{1}{p-1}$ in $\bar\Q_p$.
Write $E(\gamma x)=\sum_{m=0}^{\infty}\lambda_mx^m$
for $\lambda_m\in\Z_p[\zeta_p]$.
Note the following properties,
\begin{eqnarray}\label{E:22}
\ord_p\lambda_m&\geq&\frac{m}{p-1};
\end{eqnarray}
for $0\leq m\leq p-1$ we have,
\begin{eqnarray}\label{E:33}
\lambda_m=\frac{\gamma^m}{m!}\mbox{ and } \ord_p
\lambda_m=\frac{m}{p-1}.
\end{eqnarray}
Let $f(x):=x^d+\bar{a}_{d-1}x^{d-1}+\cdots + \bar{a}_1x$ be a
polynomial in $\F_q[x]$. Let $\hat{f}(x)$ be its Teichm\"uller lifting
in $\Omega_a$, that is,
$\hat{f}(x):=x^d+\hat{a}_{d-1}x^{d-1}+\cdots+\hat{a}_1x$ where
$\hat{a}_\ell^q=\hat{a}_\ell$ and $\bar{a}_\ell\equiv\hat{a}_\ell\bmod p$.
We let $\bar{a}_d:=1$ and $\hat{a}_d:=1$ for ease of formulation.
Let
\begin{eqnarray*}
G(X)=
(\sum_{m_1=0}^{\infty}\lambda_{m_1}\hat{a}_1^{m_1}X^{m_1})
\cdots
(\sum_{m_d=0}^{\infty}\lambda_{m_d}\hat{a}_d^{m_d}X^{dm_d})
=\sum_{n=0}^{\infty}G_nX^n.
\end{eqnarray*}
Then $G_n=0$ for $n<0$.
For every integer $n\geq 0$,
\begin{eqnarray}\label{E:3}
G_n=\sum_{\substack{m_\ell\geq 0\\\sum_{\ell=1}^{d}\ell m_\ell=n}}
\lambda_{m_1}\cdots\lambda_{m_d}\hat{a}_\ell^{m_1}
                           \cdots\hat{a}_{d-1}^{m_{d-1}},
\end{eqnarray}
where we define $0^0:=1$.
Let $\Supp(f)$ denote a set of $\ell$'s with $1\leq \ell\leq d$
such that $\bar{a}_\ell=0$ for every $\ell\not\in\Supp(f)$.
Then for $n\geq 0$,
\begin{eqnarray}\label{E:4}
G_n=\sum_{\substack{m_\ell\geq 0\\\sum_{\ell\in\Supp(f)}\ell m_\ell=n}}
\prod_{\ell\in\Supp(f)}\lambda_{m_\ell}\hat{a}_\ell^{m_\ell};
\end{eqnarray}
hence,
\begin{eqnarray}\label{E:Bound}
\ord_pG_n\geq \frac{\min(\sum_{\ell\in\Supp(f)}m_\ell)}{p-1},
\end{eqnarray}
where the minimum is taken over all integers $m_\ell\geq 0$
and $\sum_{\ell\in\Supp(f)}\ell m_\ell = n$.

Let $\phi$ be the Dwork $\phi$ operator on $\cL(c)$ defined by
$\phi(\sum B_nX^n):=\sum B_{np}X^n.$
Let $G(X)$ denote the multiplication map by $G(X)$, then
the composition map $\alpha:=\tau^{-1}\cdot\phi \cdot G(X)$
is an endomorphism of $\cL(p/(p-1))$.
Let $F$ represent the matrix of $\alpha$
under the monomial basis
$\{1,X,X^2,\ldots\}$ of $\cL(p/(p-1))$.
Note that $F=\{G_{pi-j}^{\tau^{-1}}\}_{i,j\geq 1}$.
Let $\alpha_a:=\alpha^a$ then, by a similar argument as in
\cite[Section 2]{Zhu:2},
one finds that $\alpha_a$
is represented by the matrix
$$
F_a =  FF^{\tau^{-1}}\cdots F^{\tau^{-(a-1)}}.
$$

Let $C_0=1$, and for every $n\geq 1$ let
\begin{eqnarray}\label{E:1}
C_n
&:=&
\sum_{1\leq u_1<u_2<\ldots<u_n}
\sum_{\sigma\in S_n}
\sgn(\sigma)\prod_{i=1}^{n}(F_a)_{u_i,u_{\sigma(i)}},
\end{eqnarray}
where $\sgn(\sigma)$ is the signature of the permutation $\sigma$
in the $n$-th symmetric group $S_n$.
Let $L(f/\F_q;T)$ be the $L$ function of exponential sums
of $f(x)$ over $\F_q$. By Dwork trace formula,
see \cite[(34)]{Bombieri}, we have
\begin{eqnarray*}
L(f/\F_q;T)
&=&\frac{\det(I-F_aT)}{(1-qT)\det(I-F_aq T)}\\
&=&\frac{1+\sum_{n=1}^{\infty}(-1)^nC_nT^n}
{(1-qT)(1+\sum_{n=1}^{\infty} (-1)^nC_nq^{n}T^n)}\\
&=&1+b_1T+\cdots +b_{d-1}T^{d-1}
\end{eqnarray*}
which lies in $\Z[\zeta_p][T]$.
Let the Newton polygon $\NP(f/\F_q)$
of the $L$ function $L(f/\F_q;T)$ be the lower
convex hull of points $(n,\ord_qb_n)$ for $0\leq n\leq d-1$
with $b_0:=1$.

\begin{proposition}\label{P:Dwork}
Let notation be as above.
Then $\NP(f/\F_q)$ is equal to the
lower convex hull of points
$(n,\ord_q C_n)$ where $0\leq n\leq d-1$.
Moreover, $\NP(f/\F_q)$ can be obtained from
the (normalized)  Newton polygon of the zeta function of
the curve $y^p-y=f(x)$ by reducing a factor $1/(p-1)$
the ordinates and the abscissas of the latter.
\end{proposition}
\begin{proof}
  The first assertion follows from a similar argument as that of
  Proposition 2.2 of \cite{Zhu:1}.  The second assertion is a
  well-known fact, which can be found in \cite[(106)]{Bombieri} for
example, or see \cite[Introduction]{Zhu:1}.
\end{proof}

\section{$X_1: y^7-y=x^5+cx^2$ is supersingular}

Let $\lfloor r\rfloor$ denote the greatest integer $\leq r$.
In this section let $\Supp(f)=\Supp(x^5+cx^2)=\{2,5\}$.
For any $c\in\F_q$ with $q=7^a$,
by (\ref{E:Bound}), one has
\begin{eqnarray*}
\ord_7G_n
&\geq&
\left\{
\begin{array}{ll}
+\infty                &\mbox{if $n<0$},\\
\min_{m_2,m_5\geq 0; 2m_2+5m_5=n}(m_2+m_5)/6 &\mbox{if $n\geq 0$}
\end{array}
\right.
\\
&\geq&
\left\{
\begin{array}{ll}
+\infty & \mbox{if $n<0$} \\
\frac{n}{12}-\frac{\lfloor\frac{n}{5}\rfloor-1}{4}&
                    \mbox{if $n\geq 0$ and $n\equiv 1,3\bmod 5$,}\\
\frac{n}{12}-\frac{\lfloor\frac{n}{5}\rfloor}{4}&
                    \mbox{if $n\geq 0$ and $n\equiv 0,2,4\bmod 5$,}
\end{array}\right.
\end{eqnarray*}
Let $g_{ij}$ be the lower bound of
$\ord_7G_{7i-j}$ given above.

\begin{lemma}\label{L:1}
Let $a$ be any positive integer divisible by $4$. Let
$M_1, \ldots, M_a$ be infinite matrices over a $p$-adic ring
such that every matrix $M_\ell$ (where $\ell = 1,\ldots,a$)
has  $\ord_p (M_\ell)_{ij}\geq g_{ij}$.
Then
$$\ord_p (M_1\cdots M_a)_{ij}> \frac{5a}{12}+\frac{i-j}{12}.$$
\end{lemma}
\begin{proof}
It is easy to see that it suffices to show it for $a=4$.
Let $D_{ij}:= g_{ij}-(\frac{5}{12}+\frac{i-j}{12})$.
Let $D:=\{D_{ij}\}_{i,j\geq 1}$.
Using the hypothesis,  one gets
$g_{ij}\geq \frac{7i-j}{30}$, hence
$D_{ij}>\frac{3}{6}$ whenever $\max(i,j)>12$.
Let $\ell:=12$. Then
$$
\left\{
\begin{array}{ll}
\ord_p D_{ij}\geq -\frac{1}{6} &\mbox{if $1\leq i,j\leq \ell$};\\
\ord_p D_{ij}>\frac{3}{6}     &\mbox{if $i>\ell$ or $j>\ell$}.
\end{array}
\right.
$$
(In fact, $D_{ij}\geq \frac{1}{12}$ except for
$D_{1,2}=-\frac{1}{6}$.)
Let $D*D$ be the infinite matrix defined by
$$(D*D)_{ij}:= \min_{k\geq 1}(D_{ik}+D_{kj}).$$
One has
$$
\left\{
\begin{array}{ll}
\ord_p (D*D)_{ij}\geq -\frac{2}{6}&\mbox{if $1\leq i,j\leq \ell$};\\
\ord_p (D*D)_{ij}> \frac{2}{6}    &\mbox{if $i>\ell$ or $j>\ell$};\\
\ord_p (D*D)_{ij}> \frac{6}{6}    &\mbox{if $i,j>\ell$}.
\end{array}
\right.
$$
Define $(D*D)*(D*D)$ analogously and apply the above method again,
then
\begin{eqnarray}\label{E:5}
\left\{
\begin{array}{ll}
\ord_p ((D*D)*(D*D))_{ij}\geq -\frac{4}{6}&\mbox{if $1\leq i,j\leq \ell$};\\
\ord_p ((D*D)*(D*D))_{ij}> 0              &\mbox{if $i>\ell$ or $j>\ell$};\\
\ord_p ((D*D)*(D*D))_{ij}>\frac{4}{6}     &\mbox{if $i,j>\ell$}.
\end{array}
\right.
\end{eqnarray}
On the other hand, it is not hard to show that for every $i,j\geq 1$
\begin{eqnarray}
\ord_p(M_1M_2)_{ij}-(\frac{5\cdot 2}{12}+\frac{i-j}{12})&\geq&(D*D)_{ij}.\\
\ord_p((M_1M_2)(M_3M_4))_{ij}-(\frac{5\cdot 4}{12}
+\frac{i-j}{12})&\geq&((D*D)*(D*D))_{ij}.
\label{E:2}
\end{eqnarray}
Let $\bar{D}:=\{D_{ij}\}_{1\leq  i,j\leq \ell}$.
We use a computer to verify that
$((\bar{D}*\bar{D})*(\bar{D}*\bar{D}))_{ij}>0$ for all
$1\leq i,j\leq \ell$. By (\ref{E:5}) and (\ref{E:2}) we see that
$$
\ord_p(M_1M_2M_3M_4)_{ij} > \frac{5\cdot 4}{12}+\frac{i-j}{12}.
$$
This finishes the proof.
\end{proof}

\begin{remark}
In the proof of the lemma above, the factor $1/12$ in $(i-j)/12$
will not affect the existence of $a$, though it will
probably affect the lower bound of $a$.
Namely if one chooses other factors, say $1/8$, then one may end up with
$>4$ many operations with $D$ and consequently will need to
modify the lower bound of $a$ in the statement.

In the above proof, one may intend to define
$(\bar{D}*\bar{D})_{ij}:=\min(D_{ik}+D_{kj},
\frac{3}{6}+\frac{3}{6})$.
But further analysis shows that our definition in the proof
suffices for the purpose of proving the first $\ell$ by $\ell$
submatrix of $(D*D)*(D*D)$ is positive.
\end{remark}

\begin{lemma}\label{L:2}
Let $a$ be a positive multiple of $4$ and notations be as above. Then
$$\ord_qC_1>\frac{5}{12};\quad \ord_qC_2>\frac{5}{6}.$$
\end{lemma}
\begin{proof}
Since $F=\{G_{7i-j}^{\tau^{-1}}\}_{i,j\geq 1},$
and $\ord_7 G_{7i-j} \geq g_{ij}$ as in Lemma \ref{L:1},
The hypothesis of Lemma \ref{L:1} is satisfied for
$F, F^{\tau^{-1}}, \ldots, F^{\tau^{-(a-1)}}$.
Apply Lemma \ref{L:1} to $F_a=FF^{\tau^{-1}}\cdots F^{\tau^{-(a-1)}}$,
one gets
\begin{eqnarray}\label{E:7}
ord_7(F_a)_{ij} &>& \frac{5a}{12} + \frac{i-j}{12}.
\end{eqnarray}
By (\ref{E:1}), one observes easily that $\ord_7 C_1>\frac{5a}{12}$.

On the other hand,
\begin{eqnarray*}
\ord_7C_2
&\geq &
\min_{i,j\geq 1}(\ord_p((F_a)_{ii}(F_a)_{jj}),\ord_p((F_a)_{ij}(F_a)_{ji}))
> \frac{2\cdot 5a}{12}= \frac{5a}{6},
\end{eqnarray*}
where the last inequality follows from (\ref{E:7}).
The lemma follows immediately.
\end{proof}

\begin{proposition}\label{P:1}
The curve $X_1: y^7-y=x^5+cx^2$ over $\bar\F_7$ is
supersingular.
\end{proposition}
\begin{proof}
Let $c$ be an arbitrary element in $\F_{7^a}$ for some
$a$ which is a positive multiple of $4$ (one can always do so since our
Newton polygon does not depend on $a$).
Because we know that the (normalized)
Newton polygon of any abelian variety over finite fields
is symmetric whose vertices all have integral coordinates,
the same holds for curves over finite fields.
Since $\NP(f/\F_{7^a})$ is of the same shape as
the Newton polygon of the zeta function of $X_1$ shrunk by a factor
of $1/6$, we know that $\NP(f/\F_{7^a})$ is symmetric and
every vertex has its $y$-coordinate equal to a multiple of $1/6$.
By Lemma \ref{L:2} we know that $\ord_qC_1>5/12$ and $\ord_qC_2>5/6$.
Then it is easy to derive that the
the first slope of $\NP(f/\F_{7^a})$ has to be
$1/2$ and so $X_1$ is supersingular.
\end{proof}

\section{$X_2: y^5-y = x^7+cx$ is supersingular}

In this section let $\Supp(f)=\Supp(x^7+cx)=\{1,7\}$.
For any $c\in\F_q$ with $q=5^a$, by (\ref{E:Bound}), one has
\begin{eqnarray*}
\ord_5G_n
&\geq&
\left\{
\begin{array}{ll}
+\infty                &\mbox{if $n<0$},\\
\min_{m_1,m_7\geq 0; m_1+7m_7=n}(m_1+m_7)/4 &\mbox{if $n\geq 0$}
\end{array}
\right.\\
&=&
\left\{
\begin{array}{ll}
+\infty & \mbox{if $n<0$} \\
\frac{\lfloor\frac{n}{7}\rfloor+\bar{n}}{4}&\mbox{if $n\geq 0$}
\end{array}\right.,
\end{eqnarray*}
where $\bar{n}$ is the least nonnegative residue of $n\bmod 7$.
For all $i,j\geq 1$ let $g_{ij}$ be the lower bound of
$\ord_5G_{5i-j}$ given above.

\begin{lemma}\label{L:5}
Let $a$ be any positive integer divisible by $8$. Let
$M_1, \ldots, M_a$ be infinite matrices over a $p$-adic ring
such that every matrix $M_\ell$ (where $\ell = 1,\ldots,a$)
has  $\ord_p (M_\ell)_{ij}\geq g_{ij}$.
Then
$$\ord_p (M_1\cdots M_a)_{ij}> \frac{5a}{12}+\frac{i-j}{8}.$$
\end{lemma}
\begin{proof}
Of course the proof is analogous to that of Lemma \ref{L:1}.
We shall briefly describe our proof below.
It suffices to show it for $a=8$.
Let $D_{ij}:=\frac{5}{12}+\frac{i-j}{8}$ and let $D:=\{D_{ij}\}$.
Then $D_{ij}>\frac{7}{6}$ if $\max(i,j)>27$.
In fact one verifies that
$D_{3,1}=-\frac{1}{6}$ and
$D_{2,3}=-\frac{1}{24}$
are the only negative
entries in $D$.
Let $\ell:=27$.
Let $\bar{D}:=\{D_{ij}\}_{1\leq i,j\leq \ell}$. Then
by computing
$((\bar{D}*\bar{D})*(\bar{D}*\bar{D}))*
((\bar{D}*\bar{D})*(\bar{D}*\bar{D}))$,
one notes that each $(i,j)$-th entry is $>0$.
Then we conclude the lemma with similar argument as
that of Lemma \ref{L:1}.
\end{proof}

Analogous to Lemma \ref{L:2} one shows that
$$
\ord_q C_1>\frac{5}{12},\quad
\ord_q C_2>\frac{2\cdot 5}{12},\quad
\ord_q C_3>\frac{3\cdot 5}{12}.$$
Then $X_2$ is supersingular.

\end{document}